\documentclass[12pt]{amsart}

\usepackage{amsmath, amssymb}
\usepackage{array}
\usepackage[frame,cmtip,arrow,matrix,line,graph,curve]{xy}
\usepackage{graphpap, color, paralist, pstricks}
\usepackage[mathscr]{eucal}
\usepackage[pdftex]{graphicx}
\usepackage[pdftex,colorlinks,backref=page,citecolor=blue]{hyperref}
\usepackage{tikz}
\usepackage{tikz-cd}

\usepackage{enumitem}

\newtheorem{theorem}{Theorem}[section]

\newtheorem{proposition}[theorem]{Proposition}
\newtheorem{corollary}[theorem]{Corollary}
\newtheorem{lemma}[theorem]{Lemma}

\newtheorem{conjecture}[theorem]{Conjecture}

\newtheorem{observation}[theorem]{Observation}

\theoremstyle{definition}
\newtheorem{definition}[theorem]{Definition}

\newtheorem{question}[theorem]{Question}
\newtheorem{remark}[theorem]{Remark}

\newtheorem*{acknowledgement}{Acknowledgement}

\newcommand{\FF}{\mathbb{F} }
\newcommand{\NN}{\mathbb{N} }

\newcommand{\QQ}{\mathbb{Q} }

\newcommand{\ZZ}{\mathbb{Z} }

\newcommand{\bp}{\mathbf{p} }

\title{Distribution of the number of zeros of polynomials over a finite field}
\author{Ritik Jain}
\address{Ritik Jain, Department of Mathematics, Fordham University, New York, NY 10023}
\email{rjain14@fordham.edu}

\author{Han-Bom Moon}
\address{Han-Bom Moon, Department of Mathematics, Fordham University, New York, NY 10023}
\email{hmoon8@fordham.edu}

\author{Peter Wu}
\address{Peter Wu, Department of Mathematics, Fordham University, New York, NY 10023}
\email{pwu34@fordham.edu}

\date{\today}

\begin{document}

\maketitle

\begin{abstract}
We study the probability distribution of the number of zeros of multivariable polynomials with bounded degree over a finite field. We find the probability generating function for each set of bounded degree polynomials. In particular, in the single variable case, we show that as the degree of the polynomials and the order of the field simultaneously approach infinity, the distribution converges to a Poisson distribution. 
\end{abstract}

\section{Introduction}

For some prime number $p$, take a random polynomial $f \in \ZZ/p \ZZ[x]$ with positive degree $d$. What is the average number of distinct zeros of $f$ in $\ZZ/p \ZZ$? While it is clearly somewhere between $0$ and $p$, the exact answer is not quite obvious. Interestingly, the average is \emph{always} one, regardless of the degree $d$ and the prime $p$. The authors believe this fact is well known in the algebraic combinatorics community. 

In this paper, extending the above result, we investigate the distribution of the number of distinct zeros of a random polynomial over any finite field. Let $\FF_q$ be the finite field of order $q = p^r$, where $p$ is a prime number. Fix a nonnegative integer $d$, and consider the set $\FF_q[x_1, \cdots, x_n]_{(d, \cdots, d)}$ of polynomials with $n$ variables, $\FF_q$-coefficients, and degree at most $d$ with respect to each variable. This is a $(d+1)^n$-dimensional $\FF_q$-vector space. By selecting coefficients uniformly, we choose a random polynomial $f \in \FF_q[x_1, \cdots, x_n]_{(d, \cdots, d)}$. Let $X_{n,d} : \FF_q[x_1, \cdots, x_n]_{(d, \cdots, d)} \to \NN$ be the random variable of the number of distinct zeros of $f$. The main result of this paper is the computation of its probability generating function. 

\begin{theorem}[\protect{Theorem \ref{thm:distribution}}]\label{thm:probabilitygeneratingfunction}
Let $X_{n,d} : \FF_q[x_1, \cdots, x_n]_{(d, \cdots, d)} \to \NN$ be the random variable whose value is the number of distinct zeros of a random polynomial over $\FF_q$. If $d \ge q-1$, then the probability generating function for $X_{n,d}$ is given by 
\[
    \Phi_{n,d}(t) := \sum_{k \ge 0}P(X_{n,d} = k)t^k
    = \left(1+\frac{t-1}{q}\right)^{q^n}.
\]
\end{theorem}

\begin{corollary}\label{cor:meanvariance}
Under the same assumption, the expected value $E$ and the variance $V$ of $X_{n,d}$ are
\[
    E(X_{n,d}) = q^{n-1}, \qquad V(X_{n,d}) = q^{n-1}\left(1 - \frac{1}{q}\right).
\]
\end{corollary}

We obtain the following result when $n = 1$ by taking the limit $q \to \infty$. 

\begin{corollary}\label{cor:Possion}
If $q \to \infty$, the probability distribution of $X_{1,d}$ converges to the Poisson distribution with parameter $1$.
\end{corollary}

In Section \ref{sec:average}, we compute the average number of zeros using the incidence variety method. Though the computation is not later used to calculate $\Phi_{n,d}(t)$, we included the proof because it 1) explains the anticipated expected value without complex calculation, and 2) is flexible, as it works for many variations of the sample space, and only requires a nonnegative degree $d$. See Remark \ref{rmk:totaldegreeaverage}. Section \ref{sec:singlevar} is devoted to computing $\Phi_{1,d}(t)$. Our method is based on a similar computation by Leont\'ev \cite{Leo06}, who studied the case of single-variable monic polynomials of a fixed degree. Our result (Theorem \ref{thm:pgf}) for $\Phi_{1,d}(t)$ is a careful application of his result. 

Our main contribution is in Section \ref{sec:multivar}. By employing elementary algebraic tools such as Lagrange interpolation and the Chinese remainder theorem, we show that $\Phi_{n,d}(t)$ is stable when $d \ge q-1$, then calculate it explicitly in Theorem \ref{thm:distribution}. Finally, in the last section, we discuss some related questions. 

For real polynomials, there have been numerous results on an analogous problem. If the coefficients are independent and normally distributed, a classical result of Kac asserts that the average converges to $\frac{2}{\pi}\log d$ \cite{Kac43}. See \cite{EK95} for a nice survey, its connection with geometry, and references. For more recent results, see, for example, \cite{DPSZ02, NV21}. 

For a similar question over $p$-adic fields, there have also been several results, including \cite{Eva06, BCFG22}. Since the ring of $p$-adic integers can be obtained by taking the inverse limit of $\ZZ/p^k\ZZ$, linking these results with the case of $\ZZ/p^k\ZZ$-coefficients will be interesting. See Question \ref{que:p-adic}.

\acknowledgement

P.W. was supported by the Fordham Summer Research Assistant Fellowship.


\section{The average number of zeros}\label{sec:average}

In this section, we fix the notation. Then, using the incidence variety, we show that the average number of zeros of a random polynomial with nonnegative degree is $q^{n-1}$. 

Let $\FF_q$ be the finite field of order $q = p^r$, where $p$ is a prime number. The ring of polynomials with $n$ variables and $\FF_q$-coefficients is denoted by $\FF_q[x_1, \cdots, x_n]$. 

\begin{definition}\label{def:multidegree} 
Fix an integer $d \ge 0$. The subset of $\FF_q[x_1, \cdots, x_n]$ of polynomials whose degree is at most $d$ with respect to $x_i$ for each $1 \le i \le n$ is denoted by $\FF_q[x_1, \cdots, x_n]_{(d,\cdots, d)}$. The subset of polynomials whose total degree is at most $d$ is denoted by $\FF_q[x_1, \cdots, x_n]_d$.
\end{definition}

Clearly, $\FF_q[x_1, \cdots, x_n]_d \subset \FF_q[x_1, \cdots, x_n]_{(d,\cdots, d)}$. If $n = 1$, $\FF_q[x_1]_d = \FF_q[x_1]_{(d)}$. 

Since $\FF_q[x_1, \cdots, x_n]_{(d, \cdots, d)}$ and $\FF_q[x_1, \cdots, x_n]_d$ are finite-dimensional $\FF_q$-vector spaces, we can choose a random polynomial by fixing a basis and uniformly selecting coefficients. 

\begin{definition}\label{def:zero}
A point $\bp := (a_1, \cdots, a_n) \in \FF_q^n$ is a \emph{zero} of a polynomial $f \in \FF_q[x_1, \cdots, x_n]$ if $f(\bp) = 0$.
\end{definition}

Thus, we only consider zeros lying in $\FF_q^n$.

\begin{theorem}\label{thm:affinemean}
Fix $d \ge 0$. The average number of zeros of a random polynomial in $\FF_q[x_1, \cdots, x_n]_{(d,\cdots, d)}$ is $q^{n-1}$.
\end{theorem}

\begin{proof}
Consider the \emph{incidence variety}
\[
    \mathrm{Inc} := \{(\bp, f)\;|\; f(\bp) = 0\} \subset \FF_q^n \times \FF_q[x_1, \cdots, x_n]_{(d, \cdots, d)}.
\]
Then, there are two projection maps 
\[
\xymatrix{
& \mathrm{Inc} \ar[ld]_{\pi_1} \ar[rd]^{\pi_2}\\
\FF_q^n && \FF_q[x_1, \cdots, x_n]_{(d, \cdots, d)},
}
\]
where $\pi_1(\bp, f) := \bp$ and $\pi_2(\bp, f) := f$. For each $f \in \FF_q[x_1, \cdots, x_n]_{(d, \cdots, d)}$, $|\pi_2^{-1}(f)|$ is the number of zeros of $f$. Thus, the desired average is given by 
\[
    \frac{|\mathrm{Inc}|}{|\FF_q[x_1, \cdots, x_n]_{(d, \cdots, d)}|}.
\]

First of all, the dimension of $\FF_q[x_1, \cdots, x_n]_{(d, \cdots, d)}$ is $(d+1)^n$. On the other hand, for $|\mathrm{Inc}|$, observe that, for any $\bp \in \FF_q^n$, we may define the evaluation map $\mathrm{ev}_\bp : \FF_q[x_1, \cdots, x_n]_{(d, \cdots, d)} \to \FF_q$ by setting $\mathrm{ev}_\bp (f) := f(\bp)$. Then, for each $\bp \in \FF_q^n$, $\pi_1^{-1}(\bp)$ is in bijection with $\ker \mathrm{ev}_\bp$. Since $\mathrm{ev}_\bp$ is clearly surjective, by the dimension theorem, 
\[
    |\pi_1^{-1}(\bp)| = |\ker \mathrm{ev}_\bp| = 
    q^{(d+1)^n - 1}.
\]
Since every fiber of $\pi_1$ has the same cardinality, for any $\bp \in \FF_q^n$, $|\mathrm{Inc}| = |\FF_q^n||\pi_1^{-1}(\bp)| = q^{(d+1)^n + n - 1}$.

In sum, we have 
\[
    \frac{|\mathrm{Inc}|}{\FF_q[x_1, \cdots, x_n]_{(d, \cdots, d)}} = 
    \frac{q^{(d+1)^n + n-1}}{q^{(d+1)^n}} = q^{n-1}.
\]
\end{proof}

\begin{remark}\label{rmk:totaldegreeaverage}
\begin{enumerate}
\item By similar logic, it is straightforward to show that the average number of zeros of a random polynomial in $\FF_q[x_1, \cdots, x_n]_d$ is also $q^{n-1}$. 
\item Unlike Theorem~\ref{thm:probabilitygeneratingfunction}, the result does not necessarily require $d \ge q-1$ .
\end{enumerate}
\end{remark}

Since the average in Theorem \ref{thm:affinemean} does not depend on the degree $d$, we can make the following conclusion by taking the limit $d \to \infty$. 

\begin{corollary}\label{cor:affinemean}
The average number of distinct zeros of a randomly chosen polynomial in $\FF_q[x_1, \cdots, x_n]$ is $q^{n-1}$. In particular, on average, a random polynomial in $\FF_q[x]$ has precisely one zero. 
\end{corollary}


\section{Single variable case}\label{sec:singlevar}

Here, we compute the probability generating function for the  $n = 1$ variable case. 

\begin{definition}\label{def:Xnd}
Let $X_{n,d} : \FF_d[x_1, \cdots, x_n]_{(d,\cdots, d)} \to \NN$ be the random variable of the number of distinct zeros in $\FF_q$ of a random polynomial. 
\end{definition}

The distribution of the number of zeros of a random polynomial was studied by Leont\'ev under a slightly different setup \cite{Leo06} for $n = 1$. We deduce our generating function 
\begin{equation}\label{eqn:probgenfunctionn=1}
    \Phi_{1,d}(t) := \sum_{k \ge 0}P(X_{n,d} = k)t^k
\end{equation}
from his combinatorial computation. We emphasize that the $d \ge q-1$ assumption is unnecessary for the proof we provide here for $n = 1$ variables.

Let $B_{d,q,k}$ be the number of degree $d$ \emph{monic} polynomials in $\FF_q[x]$ with $k$ distinct zeros. Using the inclusion-exclusion principle and residue calculation, Leont\'ev showed that the set of $B_{d,q,k}$ satisfies the following generating function \cite[Lemma 4]{Leo06}:
\begin{equation}\label{eqn:Leontevgenerating}
    \sum_{k \ge 0}\frac{B_{d,q,k}}{q^k}t^k = \sum_{i=0}^d\frac{{q \choose i}}{q^i}(t-1)^i.
\end{equation}

From this result, we compute the probability generating function over $\FF_q[x]_d$. We use the convention that the zero polynomial has degree $-\infty$.

\begin{lemma}\label{lem:cdqk}
Let $C_{d,q,k}$ be the number of polynomials in $\FF_q[x]_d$ with $k$ distinct zeros. Then
\[
    C_{d,q,k} = \begin{cases}
        (q-1)\sum_{e=0}^{d}B_{e,q,k}, &k \ne q,\\
        (q-1)\sum_{e=0}^{d}B_{e,q,q} + 1, &k = q.
    \end{cases}
\]
\end{lemma}

\begin{proof}
Recall that $\FF_q[x]_d$ is the vector space of all polynomials of degree at most $d$. As nonzero scalar multiplication does not change the number of zeros of a given polynomial, to count the number of all polynomials with degree at most $d$, we obtain the first formula. When $k = q$, we add one to account for the zero polynomial.
\end{proof}

\begin{theorem}\label{thm:pgf}
The probability generating function $\Phi_{1,d}(t)$ in \eqref{eqn:probgenfunctionn=1} satisfies
\begin{equation}\label{eqn:pgnn=1}
    \Phi_{1,d}(t) = \sum_{i=0}^d \frac{(q^{d+1-i}-1){q \choose i}}{q^{d+1}}(t-1)^i + \frac{t^q}{q^{d+1}}.
\end{equation}
\end{theorem}

\begin{proof}
Note that 
\[
    \Phi_{1,d}(t) = \sum_{k \ge 0}P(X_{1,d} = k)x^k = \sum_{k \ge 0}\frac{C_{d,q,k}}{q^{d+1}}t^k.
\]
By Lemma \ref{lem:cdqk}, 
\begin{equation}
\begin{split}
    \Phi_{1,d}(t) &= 
    \sum_{k \ge 0}(q-1)\sum_{e=0}^{d}\frac{B_{e,q,k}}{q^{d+1}}t^k + \frac{t^q}{q^{d+1}} \\
    &= \frac{q-1}{q}\sum_{k \ge 0}\sum_{e=0}^{d}\frac{B_{e,q,k}}{q^d}t^k + \frac{t^q}{q^{d+1}} = \frac{q-1}{q}\sum_{e=0}^{d}\sum_{k \ge 0}\frac{B_{e,q,k}}{q^{d-e}q^e}t^k + \frac{t^q}{q^{d+1}} \\
    &= \frac{q-1}{q}\sum_{e=0}^{d}\frac{1}{q^{d-e}}\left(\sum_{k \ge 0}\frac{B_{e,q,k}}{q^e}t^k\right) + \frac{t^q}{q^{d+1}}
    = \frac{q-1}{q^{d+1}}\sum_{e=0}^{d}q^{e}\left(\sum_{k \ge 0}\frac{B_{e,q,k}}{q^e}t^k\right) + \frac{t^q}{q^{d+1}}.
\end{split}
\end{equation}
Using \eqref{eqn:Leontevgenerating}, we have
\begin{equation*}
\begin{split}
    \Phi_{1,d}(t) &= \frac{q-1}{q^{d+1}}\sum_{e=0}^{d}q^e\left(\sum_{i=0}^e\frac{{q \choose i}}{q^i}(t-1)^i\right) + \frac{t^q}{q^{d+1}}
    = \frac{q-1}{q^{d+1}}\sum_{e=0}^{d}\sum_{i=0}^{e}\frac{\binom{q}{i}}{q^i}(t-1)^i q^e + \frac{t^q}{q^{d+1}} \\
    &= \frac{q-1}{q^{d+1}}\sum_{i=0}^{d}\left(\sum_{e=i}^{d}q^e\right)\frac{\binom{q}{i}}{q^i}(t-1)^i + \frac{t^q}{q^{d+1}}
    = \frac{q-1}{q^{d+1}}\sum_{i=0}^{d}\frac{q^i (q^{d+1-i}-1)}{q-1}\frac{\binom{q}{i}}{q^i}(t-1)^i + \frac{t^q}{q^{d+1}} \\
    &= \sum_{i=0}^d \frac{(q^{d+1-i}-1){q \choose i}}{q^{d+1}}(t-1)^i + \frac{t^q}{q^{d+1}}.
\end{split}
\end{equation*}
\end{proof}

From the generating function, we immediately obtain the mean and the variance of the number of zeros. 

\begin{corollary}\label{cor:meanandvar}
Let $E$ and $V$ respectively denote the mean and variance of $X_{1,d}$. Then 
\begin{equation}\label{eqn:meanvariance}
E(X_{1,d}) = 1, \qquad V(X_{1,d}) = 1 - \frac{1}{q}.
\end{equation}
\end{corollary}

\begin{proof}
Since $E(X_{1,d}) = \sum_{k \ge 0}k P(X_{1,d} = k) = \Phi_{1,d}'(1)$, and 
\[
\begin{split}
    V(X_{1,d}) &= E(X_{1,d}^2) - E(X_{1,d})^2 = \sum_{k \ge 0}k^2 P(X_{1,d}=k)- \left(\sum_{k \ge 0}kP(X_{1,d}=k)\right)^2\\
    &= \Phi_{1,d}''(1) + \Phi_{1,d}'(1) - (\Phi_{1,d}'(1))^2,
\end{split}
\]
the result immediately follows from Theorem \ref{thm:pgf} by computing the first and second derivatives of $\Phi_{1,d}(t)$ at the point $t = 1$.
\end{proof}

When $d$ is large relative to $q$, the generating function $\Phi_{1,d}(t)$ reduces to a surprisingly simple formula. 

\begin{corollary}\label{cor:pgf}
If $d \ge q-1$, the formula in \eqref{eqn:pgnn=1} simplifies to 
\[
    \Phi_{1,d}(t) = \left(1+\frac{t-1}{q}\right)^q.
\]
\end{corollary}

\begin{proof}
When $d \ge q$, since $P(X_{1,d} = k) = 0$ for $k > q$, we have
\[
    \Phi_{1,d}(t) = \sum_{i=0}^q \frac{(q^{d+1-i}-1){q \choose i}}{q^{d+1}}(t-1)^i + \frac{t^q}{q^{d+1}}.
\]
Using the binomial theorem, we obtain
\[
\begin{split}
\Phi_{1,d}(t) &= \sum_{i=0}^q{q \choose i}\left(\frac{t-1}{q}\right)^i - \sum_{i=0}^q \frac{{q \choose i}}{q^{d+1}}(t-1)^i + \frac{t^q}{q^{d+1}} \\
&= \left(1 + \frac{t-1}{q}\right)^q - \frac{(1+(t-1))^q}{q^{d+1}}+ \frac{t^q}{q^{d+1}} 
= \left(1 + \frac{t-1}{q}\right)^q.
\end{split}
\]
Now, a routine calculation shows that $\Phi_{1,q-1}(t) = \Phi_{1,q}(t)$, yielding the desired result.
\end{proof}

Corollary \ref{cor:pgf} implies the following interesting consequence, generalizing what Leont\'ev observed for monic polynomials of degree $d = q-1$.

\begin{corollary}\label{cor:poisson}
For $d \ge q-1$, as $q \to \infty$, the probability distribution of the number of zeros of a random polynomial in $\FF_q[x]_d$ converges to a Poisson distribution with parameter $1$.
\end{corollary}
\begin{proof}
The result follows immediately from 
\[
    \lim_{q \to \infty} \Phi_{1,d}(t) = \lim_{q \to \infty}\left(1+\frac{t-1}{q}\right)^q = e^{t-1}.
\]
\end{proof}

One may wonder why the probability distribution does not change if $d \ge q-1$. In the next section, we explain why this is the case in a more general context. 


\section{Multivariable case}\label{sec:multivar}

Next, we extend the computation in Section \ref{sec:singlevar} to the general case. We additionally assume that $d \ge q-1$.

Recall that the multiplicative group $\FF_q^*$ of nonzero elements in $\FF_q$ is cyclic \cite[Theorem V.5.3]{Hun80}. Thus, for every $a \in \FF_q$, we have $a^q = a$. This implies that every polynomial $f \in \FF_q[x_1, \cdots, x_n]$ in the ideal 
\begin{equation}\label{eqn:idealforzerofunctions}
    I := (x_1^q - x_1, \cdots, x_n^q - x_n)
\end{equation}
satisfies $f(\bp) = 0$ for all $\bp \in \FF_q^n$.

The following observation is essentially the Lagrange interpolation. For all $a \in \FF_q$, let 
\begin{equation}\label{eqn:Lpolynomial}
L_a(x) := \prod_{b \ne a}\frac{(x-b)}{(a-b)}.
\end{equation}
Note that $L_a(x)$ has degree $q-1$, $L_a(a) = 1$, and $L_a(b) = 0$ for all $b \in \FF_q \setminus \{a\}$. 

\begin{lemma}\label{lem:interpolation1}
Every $f \in \FF_q[x_1,..., x_n]$ can be written as
\begin{equation}\label{eqn:interpolation}
f = \sum_{a \in \FF_q}L_a(x_n)f(x_1,\cdots, x_{n-1}, a) + \sum_{i=1}^{n}(x_i^q-x_i)h_i
\end{equation}
for some $h_1, \cdots, h_n \in \FF_q[x_1,..., x_n]$, where $h_i$ is not a multiple of $(x_j^q-x_j)$ for $j < i$.
\end{lemma}

\begin{proof}
It is straightforward to check that, for every $\bp = (p_1, \cdots, p_n)\in \FF_q^n$, 
\[
    f(\bp) = \sum_{a \in \FF_q}L_a(p_n)f(p_1, \cdots, p_{n-1}, a).
\]
Thus, $f - \sum_{a \in \FF_q}L_a(x_n)f(x_1, \cdots, x_{n-1}, a) \in I$ in \eqref{eqn:idealforzerofunctions}, implying it is an $\FF_q[x_1, \cdots, x_n]$-linear combination of $\{x_i^p - x_i\}$. That is, we can find $h_1, \cdots, h_n$. Finally, we may impose the last condition by rearranging the coefficients if $h_i$ divides $x_j^q - x_j$ for some $j < i$.
\end{proof}

\begin{lemma}\label{lem:interpolation2}
For any $q$ polynomials $\{g_a\}_{a \in \FF_q} \subset \FF_q[x_1, \cdots, x_{n-1}]_{(q-1, \cdots, q-1)}$, there exists a unique $f \in \FF_q[x_1, ..., x_n]_{(q-1, \cdots, q-1)}$ such that $f(x_1,..., x_{n-1}, a) = g_a(x_1,\cdots, x_{n-1})$ for all $a \in \FF_q$.
\end{lemma}

\begin{proof}
For each $a \in \FF_q$, consider the system of congruences
\begin{equation}\label{eqn:congruence}
   f \equiv g_a \mod{(x_n - a)}. 
\end{equation}
Since the principal ideals $(x_n - a)$ are pairwisely relatively prime, we may apply the Chinese remainder theorem \cite[Theorem III.2.25]{Hun80}. By the theorem, there exists a unique $f \in \FF_q[x_1, \cdots, x_n]$ modulo the ideal 
\[
    \bigcap_{a \in \FF_q}(x_n - a) = (\prod_{a \in \FF_q}(x_n - a)) = (x_n^q - x_n)
\]
satisfying the congruence relations in \eqref{eqn:congruence}. Indeed, by using formula \eqref{eqn:interpolation} in Lemma \ref{lem:interpolation1}, we can construct such an $f \in \FF_q[x_1, \cdots, x_n]_{(q-1, \cdots, q-1)}$ explicitly. 
\end{proof}

 \begin{proposition}\label{prop:stablizeddistribution}
Suppose that $d \ge q-1$. Then, the probability distribution of the number of zeros of a random polynomial in $\FF_q[x_1, \cdots, x_n]_{(d, \cdots, d)}$ is independent of $d$. 
\end{proposition}

\begin{proof}
Combining Lemma \ref{lem:interpolation1} and Lemma \ref{lem:interpolation2}, a polynomial $f \in \FF_q[x_1, \cdots, x_n]_{(d, \cdots, d)}$ can be uniquely chosen by selecting the polynomials $\{g_a\}_{a \in \FF_q} \subset \FF_q[x_1, \cdots, x_n]_{(q-1, \cdots, q-1)}$ and $h_1, \cdots, h_n \in \FF_q[x_1, \cdots, x_n]$ with certain divisibility and degree conditions. Note that the choice of $h_1, \cdots, h_n$ does not affect the number of zeros of $f$ because they are multiplied with $x_i^q - x_i$.  Therefore, while computing the distribution of the number of zeros, we may assume that the sample space is $\FF_q[x_1, \cdots, x_n]_{(q-1, \cdots, q-1)}$. 
\end{proof}

\begin{remark}
\begin{enumerate}
\item Equivalently, we may assume that the sample space is 
\[
    (\FF_q[x_1, \cdots, x_{n-1}]_{(q-1, \cdots, q-1)})^q. 
\]
This has an important consequence -- the independence of the choice of $\{g_a\}_{a \in \FF_q}$.
\item Proposition \ref{prop:stablizeddistribution} explains why the probability generating function in Corollary \ref{cor:pgf} stabilizes. 
\end{enumerate}
\end{remark}

For the proof, we henceforth assume that $d = q-1$. For $d \ge q$, we may obtain the same result by applying Proposition \ref{prop:stablizeddistribution}.

\begin{theorem}\label{thm:distribution}
The probability generating function 
\[
    \Phi_{n,d}(t) = \sum_{k \ge 0}P(X_{n,d}=k)t^k
\]
of the number of zeros of a random polynomial $f \in \FF_q[x_1,\cdots,x_n]_{(d, \cdots, d)}$ is given by 
\[
    \Phi_{n,d}(t) = (\Phi_{1,d}(t))^{q^{n-1}} = \left(1+\frac{t-1}{q}\right)^{q^n}.
\]
\end{theorem}

\begin{proof}
We proceed by induction on $n$. The $n = 1$ case holds by Corollary \ref{cor:pgf}.

For a random polynomial $f \in \FF_q[x_1, \cdots, x_n]_{(d, \cdots, d)}$, let $N(f)$ denote the number of zeros of $f$. Then, observe that choosing $f$ is equivalent to choosing $q$ random polynomials $\{g_a\}_{a \in \FF_q} \subset \FF_q[x_1, \cdots, x_{n-1}]_{(d, \cdots, d)}$, where each $g_a = f(x_1, \cdots, x_{n-1}, a)$. By \eqref{eqn:interpolation} (with $h_i = 0$), $N(f) = \sum_{a \in \FF_q}N(g_a)$. Then, we have

\begin{equation}\label{eqn:probdecomposition}
    P(N(f) = k) = \sum_{\sum_{a \in \FF_q} k_a =k}P(N(g_a) = k_a, \; \forall a \in \FF_q).
\end{equation}
Since the choice of $\{g_a\}_{a \in \FF_q}$ are independent, the right-hand side of \eqref{eqn:probdecomposition} equals 
\[
    \sum_{\sum_{a \in \FF_q}k_a=k}\left(\prod_{a \in \FF_q}P(N(g_a) = k_a)\right).
\]

We define a multivariable generating function
\begin{equation}\label{eqn:multigenftn}
    \Phi(t_a, \;a \in \FF_q) := 
    \sum_{\sum_{a \in \FF_q}k_a = k} \left(\prod_{a \in \FF_q}P(N(g_a) = k_a)t_a^{k_a}\right).
\end{equation}
The right-hand side of (\ref{eqn:multigenftn}) is given by the Cauchy product of the coefficients of polynomials. Therefore, 
\begin{equation}\label{eqn:multigenftn2}
    \Phi(t_a, \;a \in \FF_q) = 
    \prod_{a \in \FF_q}\left(\sum_{k_a \ge 0} P(N(g_a) = k_a) t_a^{k_a}\right)
    = \prod_{a \in \FF_q}\Phi_{n-1,d}(t_a)
    = \prod_{a \in \FF_q}(\Phi_{1,d}(t_a))^{q^{n-2}}
\end{equation}
by the induction hypothesis. 

On the other hand, setting $t_a = t$ for all $a \in \FF_q$, we have 
\begin{equation}\label{eqn:multigenftn3}
\begin{split}
    \Phi(t, \cdots, t) &= \sum_{\{k_a\}_{a \in \FF_q}}
    \prod_{a \in \FF_q}P(N(g_a) = k_a)t^{\sum k_a}\\
    &= \sum_{k \ge 0}\sum_{\sum_{a \in \FF_q} k_a = k} \left(\prod_{a \in \FF_q}P(N(g_a) = k_a)\right) t^k
    = \sum_{k \ge 0}P(N(f) = k)t^k = \Phi_{n,d}(t).
\end{split}
\end{equation}
Then, combining \eqref{eqn:multigenftn2} and \eqref{eqn:multigenftn3}, 
\[
    \Phi_{n,d}(t) = \Phi(t, \cdots, t) = \left(\Phi_{1,d}(t))^{q^{n-2}}\right)^{q} = (\Phi_{1,d}(t))^{q^{n-1}} = \left(1+\frac{t-1}{q}\right)^{q^n}.
\]
\end{proof}

As in the single variable case, we may compute the mean and variance from the generating function $\Phi_{n,d}(t)$. We leave the details of the computation to the interested readers. 

\begin{corollary}\label{corr: meanandvarn}
The mean $E$ and variance $V$ of $X_{n,d}$ are given by 
\[
    E(X_{n,d}) = q^{n-1}, \qquad V(X_{n,d}) = q^{n-1}\left(1 - \frac{1}{q}\right).
\]
\end{corollary}


\section{Questions}\label{sec:questions}

We leave a few related questions, supported by numerical computations. 

In the multivariable case, as shown in Section \ref{sec:multivar}, choosing $\FF_q[x_1, \cdots, x_n]_{(d, \cdots, d)}$ enables us to use induction on the number of variables. Perhaps another natural choice of a finite sample space is $\FF_q[x_1, \cdots, x_n]_d$, the set of polynomials whose total degree is at most $d$. However, a similar argument does not work for $\FF_q[x_1, \cdots, x_n]_d$, as the numbers of zeros for the restrictions $f(x_1, \cdots, x_{n-1}, a)$ and $f(x_1, \cdots, x_{n-1}, a')$ are not independent.

\begin{question}\label{que:totaldegree}
What is the probability generating function of the number of distinct zeros of a random polynomial in $\FF_q[x_1, \cdots, x_n]_d$?
\end{question}

From numerical investigation, it seems that the probability distribution for $\FF_q[x_1, \cdots, x_n]_d$ resembles the distribution for $\FF_q[x_1, \cdots, x_n]_{(d, \cdots , d)}$. Note that this is not entirely obvious, as the ratio of the dimensions of $\FF_q[x_1, \cdots, x_n]_d$ and $\FF_q[x_1, \cdots, x_n]_{(d, \cdots, d)}$ is ${n +d -1 \choose d}/(d+1)^n$, which approaches zero as $n \to \infty$. On the other hand, from Remark \ref{rmk:totaldegreeaverage}, we know their expected values are the same.

Another natural extension is the case of polynomials over a projective space. More precisely, let $\FF_q[x_1, \cdots, x_n]_d^h$ be the $\FF_q$-vector space consisting of homogeneous polynomials of degree $d$. Let $Y_{n,d}$ be the random variable of the number of nontrivial zeros up to nonzero scalar multiplication of a random polynomial $f \in \FF_q[x_1, \cdots, x_n]_d^h$. By employing the incidence variety method in Section \ref{sec:average}, it is straightforward to check that the expected value of $Y_{n,d}$ is 

\[
E(Y_{n, d}) = \frac{\left( q^{n+1}-1 \right)\left( q^{\binom{n+d}{d}-1}-1 \right)}{\left( q-1 \right)\left( q^{\binom{n+d}{d}}-1 \right)} \approx q^{n-1}.
\]

\begin{question}\label{que:projective}
What is the probability generating function of $Y_{n,d}$?
\end{question}

Another possible line of inquiry is the distribution for a random polynomial over $\ZZ/ m \ZZ$, where $m$ may be composite. In particular, let $K_{n,d}^m$ be the random variable of the number of distinct zeros of a random polynomial in $\ZZ/ m \ZZ[x_1, \cdots, x_n]_{(d, \cdots , d)}$. 

\begin{question}\label{que:squarefree}
Suppose $m = p_1 \cdots p_r$ is the square-free product of $r$ distinct primes. What is the probability generating function of $K_{n,d}^m$?
\end{question}

\begin{remark}\label{rmk: squarefree}
By the Chinese remainder theorem, any random $f \in \ZZ/ m \ZZ[x_1, \cdots, x_n]_{(d, \cdots , d)}$ is equivalent to a random $r$-tuple of polynomials $(f_{p_{1}}, \cdots , f_{p_{r}})$, where each $f_{p_{i}}$ denotes the reduction of $f$ in $\ZZ/ p_i \ZZ[x_1, \cdots, x_n]_{(d, \cdots , d)}$.
\end{remark}

Despite this fact, computing the probability generating function of  $K_{n,d}^m$ remains a nontrivial task. Numerically, however, one can discern some statistical patterns. 

\begin{conjecture}\label{conj: squarefree}
If $m$ is square-free, then the expected value $E$ of $K_{n,d}^m$ is
\[
    E(K_{n,d}^m) = m^{n-1}.
\]

Further, the distribution is stable for $d \ge p_r - 1$, where $p_r$ is the largest prime factor of $m$.
\end{conjecture}

Another potential question is the following:

\begin{question}\label{que:primepower}
If $m$ is some prime power $p^k$, what is the probability generating function of $K_{n,d}^m$?
\end{question}

Numerically, despite similarities to Question \ref{que:squarefree}, this case presents some additional subtleties. We record our observations for particular values of $m, d$ and $n$ below with a supporting graph (Figure \ref{fig:newpic2}).

\begin{conjecture}\label{conj: primepower}
If $m$ is some prime power $p^k$, then the expected value $E$ of $K_{n,d}^m$ is also
\[
    E(K_{n,d}^m) = m^{n-1}.
\]

In the single variable case, for $d > 1$, the variance $V$ of $K_{1,d}^m$ is

\[
    V(K_{1,d}^m) = k\left( 1 - \frac{1}{p} \right).
\]

Further, the distribution stabilizes as $d$ increases.
\end{conjecture}

For $n \ge 1$ variables, one can also observe the following interesting behavior regarding the probability distribution of $K_{n,d}^m$.

\begin{observation} 
If $m$ is some prime power $p^k$, the probability that $K_{n,d}^m = \beta$ is nonzero only when $\beta$ is a multiple of $p^{n-1}$. 
\end{observation}

\begin{figure}[!ht]
  \includegraphics[scale=0.40] {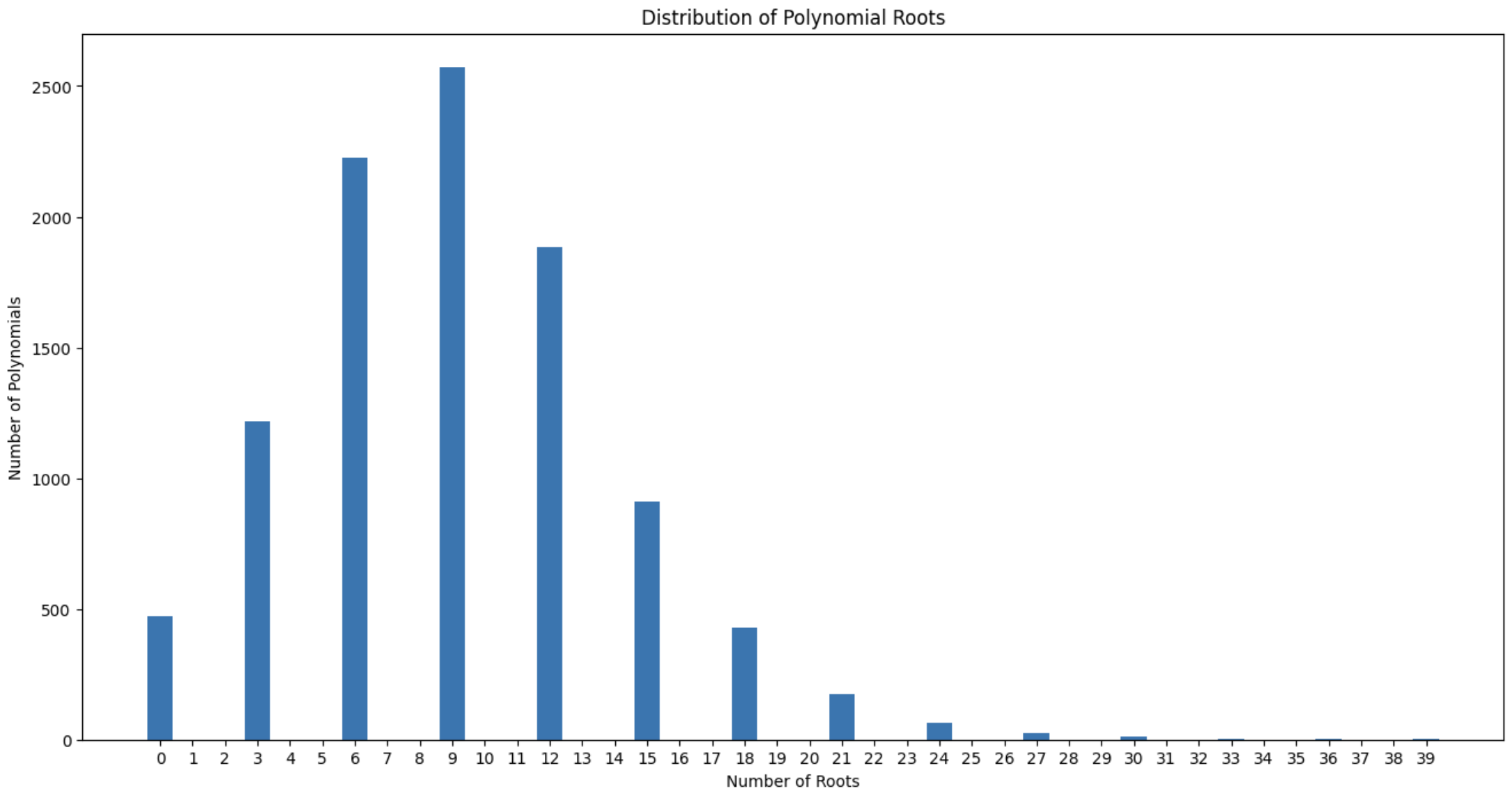}
  \caption{Distribution of zeros over $\ZZ/ 9 \ZZ[x_1,x_2]_{(2,2)}$.}
  \label{fig:newpic2}
\end{figure}

One might inquire similarly about polynomials with integer coefficients. However, it seems that obtaining an integral zero for a random polynomial in $\ZZ[x]$ is an exceedingly rare event. In other words, the expected number of integral zeros of a random polynomial in $\ZZ[x]_d$ is zero \cite{BSK20} (we encourage readers to consider the $d = 1$ case by themselves). 

Finally, another interesting question arises if we consider the $p$-adic numbers. 

\begin{question}\label{que:p-adic}
Let $\ZZ_p$ be the set of $p$-adic integers and $W_{n, d}$ be the random variable of the number of zeros of a random polynomial in $\ZZ_p[x_1, \cdots, x_n]_{(d, \cdots, d)}$. What is the probability distribution of $W_{n, d}$? One may ask a similar question for $\QQ_p$, the field of fractions of $\ZZ_p$. Can we relate this distribution to the case of $\ZZ/p^k\ZZ[x_1, \cdots, x_n]_{(d, \cdots, d)}$?
\end{question}

\end{document}